\documentclass{amsart}
\usepackage{graphicx,graphics}
\usepackage{cite}
\usepackage{pifont}
\usepackage{amsthm}
\usepackage{mathrsfs}
\usepackage{algorithm}
\usepackage{algorithmic}
\DeclareFontFamily{OML}{script}{}
\DeclareFontShape{OML}{script}{m}{it}
{ <5-20> rsfs10 }{}
\DeclareMathAlphabet{\mathscript}{OML}{script}{m}{it}

\renewcommand{\mathcal}[1]{{\mathscript #1}}
\newcommand{\red}[1]{\color{red}\huge#1}

\usepackage{array}
\usepackage{subcaption}
\usepackage{float}
\usepackage{amscd}
\usepackage{amsmath}
\usepackage{latexsym}
\usepackage{amsfonts}
\usepackage{amssymb}
\usepackage{color}
\usepackage{multicol}
\usepackage{bm}
\allowdisplaybreaks[4]
\newcommand{\dif}{\,\mathrm{d}}
\newcommand{\p}{\partial}

\usepackage[utf8]{inputenc}
\newcommand{\re}[1]{\mbox{\rm$($\ref{#1}$)$}}

\newcommand{\Rmnum}[1]{\uppercase\expandafter{\romannumeral #1}}
\newcommand\vep{{\varepsilon}}

\ifx\red\undefined 
\newcommand{\red}{\color{red}}

\fi

\newtheorem{thm}{Theorem}[section]

\newtheorem{lem}{Lemma}[section]

\usepackage{verbatim}
\usepackage{algorithm}
\usepackage{algorithmic}

\numberwithin{equation}{section}

\title{Two neural-network-based methods for solving obstacle problems}
\author{Xinyue Evelyn Zhao}  
\address{Department of Mathematics\\
Vanderbilt University, Nashville, TN 37212, USA \\ xinyue.zhao@vanderbilt.edu}
\author{Wenrui Hao}
\address{Department of Mathematics\\
 The Pennsylvania State University, University Park, PA 16802, USA \\ wxh64@psu.edu}
 \author{Bei Hu}
 \address{Department of Applied and Computational Mathematics and Statistics\\
University of Notre Dame, Notre Dame, IN 46556, USA \\ b1hu@nd.edu}
\date{}

\begin{document}
\maketitle

\begin{abstract}
Two neural-network-based numerical schemes are proposed to solve the classical obstacle problems. The schemes are based on the universal approximation property of neural networks, and the cost functions are taken as the energy minimization of the obstacle problems. We rigorously prove the convergence of the two schemes and derive the convergence rates with the number of neurons $N$. In the simulations, we use two example problems (1-D \& 2-D) to verify the convergence rate of the methods and the quality of the results.
\end{abstract}

\section{Introduction}
The classical obstacle problem is one of the simplest and yet most important free boundary problems. It models the equilibrium shape of an elastic membrane whose boundary is held fixed, and which is pushed by an obstacle from one side. The classical obstacle problem can be naturally and directly interpreted as an energy minimization problem (i.e., the classical elastic energy of the membrane) with the addition of a constraint (i.e., the solution lies above the obstacle). 
Applications of the problem can be found in many aspects, such as fluid filtration in porous media, elasto-plasticity, optimal control, and financial mathematics.

In modern computational mathematics and engineering, it is no longer extremely difficult to numerically solve the obstacle problems, as shown in numerous publications; see \cite{arakelyan2014numerical,hoppe1994adaptive,graser2009multigrid,chan2013generalized,tran20151,lee2017accurate}, for example. However, most of these research studies lack convergence proofs for the methods. From a mathematical point of view, the most challenging part of the obstacle problem is the regularity of the solution. It is well known that the {\em optimal
regularity of } the solution to the classical obstacle problem \re{model} 
is $W^{2,\infty}$; in general, this solution cannot be $C^2$. Without the continuity of second-order and higher-order derivatives, it is not an easy job to prove the convergence. Moreover, as a common challenge in all free boundary problems, the numerical solution can easily involve a large error near the free boundary, hence careful considerations are needed to deal with the free boundary.

In this work, we are interested in the design and analysis of neural-network-based methods to solve the classical obstacle problem. Over the last few decades, neural-network-based methods have been used a lot in solving various partial differential equations \cite{dissanayake1994neural,lagaris1998artificial,zhao2021convergence,sirignano2018dgm,wang2021deep}. By using the universal approximation of neural networks \cite{hornik1991approximation,cybenko1989approximation}, the numerical method of solving a partial differential equation can be transferred to an unconstrained minimization problem, which is %extremely
easily implemented. In general, the objective function is chosen as the $L^2$ norm of the equation, together with   the $L^2$ norm of all the boundary conditions, 
and then higher order of the regularity of the solution is required to prove the convergence of the method. 
In the classical obstacle problem, however, the solution is not smooth, hence there are several questions to be considered:
\begin{itemize}
    \item How to design a proper objective function such that the regularity of the solution is enough to prove the convergence?
    \item How to incorporate the boundary condition without a penalty term in the objective function?
\end{itemize}
As is mentioned, the classical obstacle problem is naturally related to an energy minimization problem, thus we will present two neural-network-based methods in which the variational form of the obstacle problem is taken as the objective function.

The paper is organized as follows. The model and the construction of neural networks are introduced in Section 2. In Section 3, we propose two neural-network-based methods and establish their convergences. Some simulation results are presented in Section 4.

\section{The model and nerual networks}
\subsection{The classical obstacle problem}
We consider the following obstacle problem
\begin{equation}\label{model}
    \begin{array}{ll}
        -\Delta u \ge f \hspace{2em} &\text{in } \Omega,\\
        u\ge \phi &\text{in } \Omega,\\
        (-\Delta u - f)(u-\phi) = 0 \hspace{2em} &\text{in } \Omega,\\
        u=0 &\text{on } \p \Omega,
    \end{array}
\end{equation}
where $\Omega \subset \mathbb{R}^p$, $p=1,2$, $\phi$ is the obstacle function satisfying $\phi < 0$ on $\p \Omega$ and $f$ typically represents external force for physical problems. The unknown function $u$ models the equilibrium position of an elastic membrane.

In the region where the membrane is below the obstacle, the solution $u$ satisfies an elliptic PDE $-\Delta u = f$; while in the remaining region, which is called the contact region, the membrane coincides with the obstacle. The set
$I = \{x: u(x) = \phi(x) \}$
is called the contact set, and the free boundary to this problem is 
$\Gamma = \p I = \p \{x: u(x) = \phi(x) \}.$ It is clear that in the contact region $-\Delta \phi \ge f$.

It is well known that the solution $u$ for \re{model} is unique and can be actually obtained via a variational formulation, namely,
\begin{equation}\label{variation1}
    \min_{u\in K} \int_\Omega \Big(\frac12 |\nabla u|^2 - uf\Big) \dif x,
\end{equation}
where $K = \{u\in H_0^1(\Omega)|u(x)\ge  \phi(x)\text{ for }x\in\Omega\}$.

For simplicity, we list the following regularity result of \re{model}, and we want to emphasize that $W^{2,\infty}$ is \textit{the optimal regularity} for this problem.

\begin{lem}{(Chapter 1, Theorems 4.1 and 4.3 in \cite{friedman2012variational})}\label{lem2.1}
If $\p \Omega \in C^{2+\alpha}$, $f\in C^\alpha(\overline{\Omega})$, $\phi\in C^2(\overline{\Omega})$ with $\phi<0$ on $\p\Omega$, then the solution $u$ of \re{model} belongs to $W^{2,\infty}%_{\text{loc}}
(\Omega)$.
\end{lem}

\subsection{Neural network discretization}
In this work, we will utilize neural networks to solve \re{model}. Generally speaking, a deep neural network with $l$ layers can be written as
\begin{equation*}
U(x;\theta) = h_{W_l,b_l} \circ \sigma \circ h_{W_{l-1},b_{l-1}} \circ \sigma \circ \cdots \circ \sigma \circ h_{W_1,b_1},
\end{equation*}
where $h_{W_i,b_i}(x) = W_i^T x + b_i$ is a linear function of $x$, $\sigma$ is the activation function (for example, sigmoid function), and $\theta = \{W_1,\cdots,W_l,b_1,\cdots,b_l\}$ are the parameters of the neural network model. 
Particularly, we will consider the two-layer neural network, which takes the following simple form
\begin{equation} \label{nn}
    U(x;\theta) = W_2^T \sigma(W_1^T x + b_1) + b_2.
\end{equation}

\section{Two neural-network-based schemes to solve \re{model}}

\subsection{Method 1}
\subsubsection{Scheme setup}
The first method follows directly from the variational form \re{variation1}. Since the set $K$ contains all the $H_0^1(\Omega)$ functions which are above the obstacle $\phi(x)$, we define
\begin{equation}
    \label{nm1-1}
    \delta_U = \max_{x\in \overline{\Omega}}\Big(\frac{\phi(x)}{\zeta(x)} - U(x;\theta)\Big)^+,
\end{equation}
and approximate $u$ as
\begin{equation}
\label{nm1-2}
    u \approx (U(x;\theta)+\delta_U)\zeta(x),
\end{equation}
where $U(x;\theta)$ is defined as the neural network approximation
\begin{equation}
    \label{nnm1new}
    U(x;\theta) = W_2^T \sigma(W_1^T x + b_1) + b_2,
\end{equation}
$\theta = \{W_1,W_2,b_1,b_2\}$ are the parameters in the neural network and $\zeta(x)$ is a smooth  cutoff function satisfying
\[
 \zeta(x) = 0 \mbox{ on } \p\Omega, \hspace{2em} \zeta(x)>0 \mbox{ in } \Omega,  \hspace{2em} |\nabla\zeta(x)|\neq 0 
      \mbox{ on } \p\Omega.
\]
Since $\delta_U$ is defined after taking $\max$ over $\overline{\Omega}$, $\delta_U$ is a constant with respect to $x$, and hence
\begin{equation}
    \label{deltaU}
    \nabla \delta_U = 0;
\end{equation}
in addition, 
\begin{equation}
\label{deltaU1}
    \nabla_\theta \delta_U = \Big(\frac{\p \delta_U}{\delta W_1},\frac{\p \delta_U}{\p b_1}, \frac{\p \delta_U}{\delta W_2},\frac{\p \delta_U}{\delta b_2}\Big)^T = \left\{
        \begin{array}{ll}
        \nabla_\theta \delta_U(x^*;\theta) \hspace{2em}&\text{ if }  \max\limits_{x\in\overline \Omega}\Big( \frac{\phi }{\zeta}- \frac{u }{\zeta }  \Big)^+ =    \frac{\phi(x^*) }{\zeta(x^*)}- \frac{u(x^*) }{\zeta(x^*) } > 0,\\
        0 &\mbox{ otherwise. }
        \end{array}
        \right.
\end{equation}

Based on our definitions, 
\begin{equation*}
\begin{split}
    (U(x;\theta)+\delta_U)\zeta(x) &\ge \Big(U(x;\theta) + \Big(\frac{\phi(x)}{\zeta(x)} - U(x;\theta)\Big)^+ \Big)\zeta(x)\\
    &\ge U(x;\theta)\zeta(x) + \Big(\phi(x) - U(x;\theta)\zeta(x)\Big)^+\\
    &=\max\Big(U(x;\theta)\zeta(x), \phi(x)\Big)  \ge \phi(x),
    \end{split}
\end{equation*}
hence the term in \re{nm1-2} which we use to approximate $u$ is definitely in the set $K$. According to the variational form \re{variation1}, we let the objective function to be
\begin{equation}
    \label{opt1}
    F_1(\theta) = F_1(X(\cdot;\theta)) =  \int_\Omega \Big(\frac12 \Big|\nabla \Big(\big(U(x;\theta)+\delta_U\big)\zeta(x)\Big) \Big|^2 - \big(U(x;\theta)+\delta_U \big)\zeta(x) f \Big) \dif x,
\end{equation}
and we are going to solve the optimization problem 
\begin{equation}\label{opp}
    \min_\theta F_1(\theta).
\end{equation}
By simple calculations and \re{deltaU}
\begin{equation}
\label{m1-1}
\begin{split}
\nabla \Big(\big(U(x;\theta)+\delta_U\big)\zeta(x)\Big) = \Big(\nabla U  + \nabla \delta_U  \Big)\zeta + \Big(U+\delta_U\Big)\nabla \zeta = \zeta\nabla U   + \Big(U+\delta_U\Big)\nabla \zeta.
    \end{split}
\end{equation}
In order to solve the optimization problem \re{opp}, we use the gradient descent method, as shown in Algorithm \ref{alg1}. The descent direction is computed as
\begin{equation}
\label{m1-2}
\begin{split}
    G_1(\theta) &\triangleq \nabla_\theta F_1(\theta)\\
    &=  \int_\Omega \nabla \Big(\big(U(x;\theta)+\delta_U\big)\zeta\Big) \cdot \nabla_\theta\Big(\nabla \Big(\big(U(x;\theta)+\delta_U\big)\zeta(x)\Big)\Big)  - \nabla_\theta \Big(U(x;\theta)+\delta_U \Big)\zeta f \dif x, 
    \end{split}
\end{equation}
where, by \re{m1-1}, 
\begin{eqnarray*}
    && \nabla_\theta\Big(\nabla \Big(\big(U(x;\theta)+\delta_U\big)\zeta(x)\Big)\Big) = \nabla_\theta\Big(\zeta\nabla U   + \big(U+\delta_U\big)\nabla \zeta\Big) = \zeta\nabla_\theta(\nabla U)  + \nabla_\theta U\nabla \zeta + \nabla_\theta \delta_U \nabla \zeta,
\end{eqnarray*}
and $\nabla_\theta \delta_U$ is computed %defined 
in \re{deltaU1}.

\begin{algorithm}[H]
\caption{Gradient descent method for solving \re{opt1}}
\label{alg1}
\begin{algorithmic}
\STATE Choose an initial guess $\theta_1$
\FOR{$k=1,2,\cdots$}
\STATE Compute the descent direction $G_1(\theta) = \nabla_\theta F_1(\theta)$;
\STATE Choose a stepsize $\alpha_k > 0$;
\STATE Set the new iterate as $\theta_{k+1} = \theta_k -\alpha_k G_1(\theta_k) $;
%\STATE Generate $m$ random points ${\bm \hat{\bm \theta}}_k = (\hat{\theta}_{k,i})_{i=1}^m$ from $[0,2\pi]$;\ %according to the probability density $\nu$;\
%\STATE Calculate the loss function $F(\mathcal{X}_k, t)$;\
%\STATE Compute a stochastic vector $G(\mathcal{X}_k, {\bm \hat{\bm \theta}}_k) = \nabla_\mathcal{X} F(\mathcal{X}_k,{\bm \hat{\bm \theta}}_k)$;\
%\STATE Set the new iterate as $\mathcal{X}_{k+1} = \mathcal{X}_k - \alpha_n G(\mathcal{X}_k,{\bm \hat{\bm \theta}}_k)$;
\ENDFOR
\end{algorithmic}
\end{algorithm}

\subsubsection{Convergence theorem}
It is well known that neural networks can uniformly approximate any continuous function of $n$ real variables with support in the unit cube \cite{hornik1991approximation,cybenko1989approximation}. Using this universal approximation property of neural networks, we have, for any small $\delta_1>0$, there is a sum of the form \re{nnm1new} such that
\begin{equation}
    \Big|U(x;\theta) - \frac{u(x)}{\zeta(x)}\Big| < \delta_1 \hspace{2em} \text{for all } x\in \Omega,
\end{equation}
which implies $0\le \delta_U \le \delta_1$. It then follows that
\begin{equation}
    \Big|U(x;\theta)+\delta_U-\frac{u(x)}{\zeta(x)}\Big| \le \Big|U(x;\theta) - \frac{u(x)}{\zeta(x)}\Big| + \delta_U \le 2\delta_1 \hspace{2em} \text{for all } x\in \Omega.
\end{equation}
Since $\zeta$ is smooth, we can always find a constant $C$, which is independent of $\delta_1$, such that
\begin{equation}
    \big|\big(U+\delta_U\big)\zeta - u\big|\le C\delta_1 \hspace{2em} \text{for all } x\in \Omega,
\end{equation}
which implies the following convergence theorem.
\begin{thm}\label{thmm1}
Let $\sigma$ be continuous, bounded, and non-constant, then for any $\delta_1 > 0$, there exists a function $U$ implemented by the neural network approximation such that
\begin{equation*}
    \big|\big(U+\delta_U\big)\zeta - u\big| \le C\delta_1,
\end{equation*}
where the constant $C$ does not depend upon $\delta_1$. 
\end{thm}

\subsubsection{Approximation rate}
In recent years, there are many quantitative estimates on the approximation rate of neural networks \cite{lu2020deep,siegel2020approximation, xu2020finite}. Using these results, we can also quantify the approximation rate of our method. To do that, we need higher order estimates for the universal approximation of neural networks. By \textit{Theorem 3} of \cite{hornik1991approximation}, we know that if the activation function $\sigma\in C^1(\mathbb{R})$ is nonlinear and bounded, then the space of all neural networks with the form \re{nnm1new} is dense in $C^1(\mathbb{R})$. Since we restrict $\phi < 0$ on $\p \Omega$, we have $u\in C^{2+\alpha}$ in a neighborhood of $ \p \Omega$ if $f\in C^\alpha(\overline{\Omega})$ then $\frac{u}{\zeta} \in C^1(\bar\Omega)$. Therefore, for any small $\delta_2 > 0$,  there is a sum of the form \re{nn} such that
\begin{equation}
    \Big\| U - \frac{u}{\zeta}\Big\|_{C^1(\overline{\Omega})} \le \delta_2.
\end{equation}
Based on the above estimate, we further have
$$0 \le \delta_U \le \delta_2+  \max_{x\in\overline \Omega}\bigg( \frac{\phi }{\zeta}- \frac{u }{\zeta }  \bigg)^+   \le \delta_2,
$$
and recall that $\delta_U$ is a constant in $x$. Clearly,
 $$
   \Big\| U  + \delta_U - \frac{u}{\zeta}\Big\|_{C^1(\overline{\Omega})} \le 2\delta_2. 
 $$
Again, we use the smoothness of the cutoff function $\zeta$ to derive
\begin{equation}
    \label{m1-11}
    \|(U+\delta_U)\zeta - u\|_{C^1(\overline{\Omega})} \le C\delta_2,
\end{equation}
where the constant $C$ is independent of $\delta_2$. Therefore, we have the following Theorem.

\begin{thm}\label{lemm1-2}
If $\p \Omega \in C^{2+\alpha}$, $f\in C^\alpha(\overline{\Omega})$, $\phi\in C^2(\overline{\Omega})$, and assume that the activation function $\sigma\in C^1(\mathbb{R})$ is nonlinear and bounded, then for any $\delta_2 > 0$, there exists a function $U$ implemented by the neural network approximation \re{nnm1new} such that
\begin{equation}\label{lem2eq11}
    \|(U+\delta_U)\zeta - u\|_{C^1(\overline{\Omega})} \le C\delta_2,
\end{equation}
where the constant $C$ does not depend upon $\delta_2$.
\end{thm}

Furthermore, combining the above result with \textit{Corollary 1} from \cite{siegel2020approximation}, we can relate the approximation rate with the number of the neurons if the activation function is $\text{ReLu}^k$ (we will use $\text{ReLu}^2$ as our activation function in the later simulations). Therefore, we have the following theorem:

\begin{thm}\label{ratem1}
If the assumptions in Theorem \ref{lemm1-2} hold and 
let $\sigma(x) = ReLu^2(x) = \max\{0,x\}^2$, then there exists a function $U$ defined in \re{nnm1new} such that
\begin{equation*}
    \|u_{\vep_1} - u \|_{H^1(\Omega)} \le C N^{-\frac12}\|u\|_{B^2},
\end{equation*}
where the constant $C$ is independent of $N$, and $\|u\|_{B^2} = \int_{\mathbb{R}^d} (1+|\omega|)^2 |\hat{u} (\omega)|\dif \omega$ denotes the Barron norm of the function $u$.
\end{thm}

\subsection{Method 2}
\subsubsection{The penalized system} 
The second method starts from a different approach: instead of directly analyzing the system \re{model}, we consider a penalized system, and a penalized term is added when the constraint $u\ge\phi$ is violated. 
To do that, we take $\beta_1$ such that
\begin{equation*}
    \beta_1 \in C^1(-\infty,\infty),\quad \beta_1'\ge 0, \quad \beta_1'' \ge 0,
\end{equation*}
an example of $\beta_1$ can be 
\begin{equation*}
\begin{split}
    &\beta_1(s) = \left\{
        \begin{array}{ll}
            s-1\hspace{2em} &\text{if } s \ge 2,\\
            \frac14 s^2 \hspace{2em} &\text{if } 0\le s\le 2,\\
            0\hspace{3.7em}&\text{if } s\le 0.
        \end{array}\right.
\end{split}
\end{equation*}
We then let $\beta_\vep(s)=\beta_1(s/\vep)$, where $\vep>0$, and consider the following approximation system, 
\begin{equation}\label{model2}
    \begin{array}{ll}
         -\Delta u_{\vep} - f =\beta_{\vep}(\phi-u_{\vep})  \hspace{2em}&\text{in } \Omega,  \\
         u_{\vep}=0 &\text{on } \p \Omega.
    \end{array}
\end{equation}
The above system is called the penalized system of \re{model}.

Correspondingly, there is also a variational formulation for the penalized system \re{model2}, i.e.,
\begin{equation}\label{variation2}
    \min_{u_{\vep} \in H_0^1(\Omega)} \int_\Omega \Big(\frac12 |\nabla u_{\vep}|^2 - u_{\vep} f + B_{\vep} (\phi-u_{\vep})\Big)\dif x,
\end{equation}
where $B_{\vep}(s) = \int_0^s \beta_{\vep}(\xi)\dif \xi$.

It can be shown that $u_{\vep}$ approximates $u$ well when ${\vep}$ tends to 0, and the difference can be bounded by the following lemma.

\begin{lem}\label{lem1}
If $\p \Omega\in C^{2+\alpha}$, $f\in C^\alpha(\overline{\Omega})$, $\phi\in C^2(\overline{\Omega})$, then system \re{model2} admits a unique solution $u_{\vep}\in C^{2+\alpha}(\overline{\Omega})$. Furthermore, for $C^*=\|(-\Delta \phi -f)^+\|_{L^\infty}$, the following estimate holds:
\begin{equation}
    \label{ueps}
    0\le u - u_{\vep} \le (C^*+1){\vep}\hspace{2em} \text{in }\Omega.
\end{equation}
\end{lem}
\begin{proof}
For system \re{model2}, we can use Schauder theory to prove that there exists a unique solution $u_{\vep} \in C^{2+\alpha}(\overline{\Omega})$. In order to prove \re{ueps}, we divide the proof into three steps.

\noindent \textbf{Step 1: } Since $u\ge \phi$, we have $\beta_{\vep}(\phi-u) \equiv 0$. Therefore,
\begin{eqnarray*}
        &-\Delta u - f - \beta_{\vep}(\phi-u) = -\Delta u - f \ge 0\hspace{2em}&\text{in } \Omega,\\
        &u = 0 &\text{on } \p\Omega.
\end{eqnarray*}
On the other hand, the equation for $u_{\vep}$ is
\begin{eqnarray*}
        &-\Delta u_{\vep} - f - \beta_{\vep}(\phi-u_{\vep}) =  0\hspace{2em}&\text{in } \Omega,\\
        &u_{\vep} = 0 &\text{on } \p\Omega.
\end{eqnarray*}
Since $\beta_{\vep}'\ge0$, we apply the maximum principle to conclude
$$u\ge u_{\vep} \hspace{2em}\text{in }\Omega,$$
and it proves the left-hand estimate in \re{ueps}.

\noindent \textbf{Step 2: } Let $z(x)=\beta_{\vep} (\phi-u_{\vep})$. Obviously, $z(x)\ge 0$ for $x\in\Omega$ and $z(x)|_{\p \Omega} = 0$. Assume that $z(x)$ takes a maximum at an interior point $x=x_0\in\Omega$, then $\phi-u_{\vep}$ also takes maximum at the same point, and hence $\Delta (\phi-u_{\vep})(x_0) \le 0$. It then follows that
\begin{equation*}
    z(x_0) = \beta_{\vep}(\phi- u_{\vep})(x_0) = -\Delta u_{\vep}(x_0) - f(x_0) \le -\Delta \phi(x_0) - f(x_0). 
\end{equation*}
Therefore, we derive that, for $x\in \overline{\Omega}$,
$$0\le \beta_{\vep}(\phi-u_{\vep}) \le -\Delta \phi(x_0) - f(x_0).$$
By the definition of $\beta_{\vep}$, we conclude
$$\phi-u_{\vep} \le (C^*+1){\vep},$$
where $C^* = \|(-\Delta \phi -f)^+\|_{L^\infty}$. Hence we have, for $x\in\overline{\Omega}$,
$$u_{\vep} + (C^*+1){\vep} \ge \phi.$$

\noindent \textbf{Step 3: } In order to prove the right-hand inequality of \re{ueps}, we need to compare $u$ with $u_{\vep} + (C^*+1){\vep}$. 

Since $u\ge \phi$, we first consider the case when $u=\phi$. Based on Step 2, we clearly obtain
$$u_{\vep} + (C^*+1){\vep} \ge \phi = u.$$

It remains to show $u_{\vep} + (C^*+1){\vep} \ge u$ in the region $x\in \Omega \,\cap\, \{u>\phi\}$, and we shall apply the maximum principle to prove it. Obviously, in $x\in \Omega \,\cap\, \{u>\phi\}$,
\begin{equation}
    \label{eqnu}
    -\Delta(u_{\vep} + (C^*+1){\vep}) - f = -\Delta u_{\vep} - f = \beta_{\vep}(\phi-u_{\vep}) \ge 0 = -\Delta u - f.
\end{equation}
Next, we need to consider the values of $u$ and $u_{\vep} + (C^*+1){\vep}$ on the boundary of $x\in \Omega \,\cap\, \{u>\phi\}$. Since $\p\{x\in \Omega \,\cap\, \{u>\phi\}\} = \{x\in \p\Omega \,\cap\, \{u>\phi\}\} \,\cup\, \{x\in\Omega \,\cap\, \p\{u=\phi\}\}$, we shall consider the boundary values on these two parts. When $ x\in \p\Omega \,\cap\, \{u>\phi\}$, we know $u = 0$, and $u_{\vep} + (C^*+1){\vep} \ge (C^*+1){\vep} \ge 0$, which implies
$$u_{\vep} + (C^*+1){\vep} \ge u\hspace{2em} x\in \p\Omega \,\cap\, \{u>\phi\}.$$
When $x\in\Omega \,\cap\, \p\{u=\phi\}$, we have $u=\phi$ and $u_{\vep} + (C^*+1){\vep} \ge \phi$ due to the Step 2. Hence, 
$$u_{\vep} + (C^*+1){\vep} \ge u\hspace{2em} x\in \Omega \,\cap\, \p\{u=\phi\}.$$
Combining with \re{eqnu}, we apply the maximum principle to derive 
$$u_{\vep} + (C^*+1){\vep} \ge u\hspace{2em} x\in \Omega,$$
and the proof is complete.
\end{proof}

\subsubsection{The neural network approximation}
Based on machine learning techniques, we approximate the solution $u_{\vep}$ to \re{model2} by a single hidden layer neutral network:
\begin{eqnarray}
    U(x;\theta) & = & W_2^T \,\sigma(W_1^T x + b_1) + b_2, \label{nn1}\\
     u_{\vep} & \approx &  U(x;\theta) \zeta(x), \label{nn2}
     \end{eqnarray}
where $\zeta$ is a smooth cutoff function such that
\[
 \zeta(x) = 0 \mbox{ on } \p\Omega, \hspace{2em} \zeta(x)>0 \mbox{ in } \Omega,  \hspace{2em} |\nabla\zeta(x)|\neq 0 
      \mbox{ on } \p\Omega.
\]
In the neural network, $\{W_1,W_2\}$ and $\{b_1,b_2\}$ are the weights and bias of the network, and $\sigma$ is the activation function such as ReLu function and sigmoid function. For simplicity, we denote the set of all parameters as $\theta = \{W_1,W_2,b_1,b_2\}$.

Based on the variational formulation \re{variation2}, we consider the objective function as
\begin{equation}
    \label{loss}
    F(\theta) = F(U(\cdot;\theta)\zeta(\cdot)) = \int_\Omega \Big(\frac12|\nabla (U\zeta)|^2 - U\zeta f + B_\vep(\phi-U\zeta)\Big) \dif x,
\end{equation}
and the problem is reformulated as the following optimization problem
\begin{eqnarray}
    \label{min}
    \min_{\theta} F(\theta),
\end{eqnarray}
which can also be solved by the gradient descent method. The algorithm is very similar as Algorithm \ref{alg1} except that 
\begin{equation}
    G_2(\theta) = \nabla_\theta F_2(\theta) = \int_\Omega \Big( \nabla(U\zeta)\cdot \nabla_\theta\big(\nabla(U\zeta)\big) - \nabla_\theta (U \zeta f) + \nabla_\theta\big(B_\vep(\phi-U\zeta)\big) \Big) \dif x, 
\end{equation}
where
\begin{eqnarray*}
    && \nabla(U\zeta) = \zeta\nabla U   + U\nabla\zeta,\\
    && \nabla_\theta\big(\nabla(U\zeta)\big) =\zeta \nabla_\theta(\nabla U)   + \nabla_\theta U \nabla \zeta,\\
    && \nabla\big(B_\vep(\phi- U\zeta)\big) = -\beta(\phi-U\zeta)\nabla_\theta U \zeta.
\end{eqnarray*}

\subsubsection{Convergence theorem}
Similar as in the analysis of Method 1, we can use the universal approximation theorem \cite{hornik1991approximation,cybenko1989approximation} and the smoothness of $\zeta$ to obtain that, for any small $\delta_3 >0$, 
\begin{equation}\label{m2a}
    \big|U\zeta - u_{\vep}\big| \le C\delta_3 \hspace{2em} \text{for all } x\in \Omega,
\end{equation}
where the constant $C$ is independent of $\delta_3$ and ${\vep}$.

After we show \re{m2a}, we can combine it with Lemma \ref{lem1}, hence we have
\begin{equation}\label{nne}
    \big|u - U\zeta\big| \le \big|u-u_{\vep}\big| + \big|u_{\vep} - U\zeta\big| \le (C^*+1){\vep} + C\delta_3 \le C\delta_3,
\end{equation}
which indicates that the difference between the unknown function $u$ and the solution solved by using neural networks can be bounded by $C\delta_3$ (or $C{\vep}$). Therefore, we have the following convergence theorem for Method 2:
\begin{thm}\label{thmm2}
If $\p \Omega\in C^{2+\alpha}$, $f\in C^\alpha(\overline{\Omega})$, $\phi\in C^2(\overline{\Omega})$, and 
let $\sigma$ be continuous, bounded, and non-constant, then for any $\delta_3 > 0$ and $\vep = \delta_3$, there exists a function $U$ implemented by \re{nn1} such that
\begin{equation*}
    \big|u - U\zeta\big| \le C\delta_3,
\end{equation*}
where the constant $C$ does not depend upon $\delta_3$ (and $\vep$). 
\end{thm}

\subsubsection{Approximation rate}
Similarly, we can also obtain the approximation rate for the Method 2. 
To do that, the $L^\infty$ estimate in \re{ueps} is not enough, and we shall work on higher order estimates of $u-u_{\vep}$, which are included in the Lemma \ref{lem2}.

\begin{lem}
If $u$ is a solution to the problem \re{model}, then for any $v \in K$, the following inequality holds
\begin{equation}\label{lemn1}
\int_\Omega \nabla u \cdot \nabla(v- u) \dif x \ge \int_\Omega f\cdot (v-u) \dif x.
\end{equation}
\end{lem}
\begin{proof}
From \re{variation1}, we know that $J[u] = \min\limits_{v\in K} J[v]$ where $J[u] = \int_\Omega \Big(\frac12 |\nabla u|^2 - uf\Big) \dif x$. For any $v\in K$ and $\tau\in[0,1]$, $(1-\tau)u + \tau v = u + \tau(v-u)$ is also in $K$, hence $J[u+\tau(v-u)] \ge J[u]$. If we denote $\widetilde{F}(\tau) = J[u+\tau(v-u)]$, we immediately have $\widetilde{F}'(0) \ge 0$, which directly leads to inequality \re{lemn1}.
\end{proof}

\begin{comment}
\begin{lem}{(Theorem 4.1 in \cite{friedman2012variational})}
If $\p \Omega \in C^{2+\alpha}$, $f\in C^\alpha(\overline{\Omega})$, $\phi\in C^2(\overline{\Omega})$, then the solution $u$ of \re{model} belongs to $W^{2,\infty}_{\text{loc}}(\Omega)$.
\end{lem}
\end{comment}

The following lemma gives $H^1$ estimate, with the constant $C$ given explicitly. 
We should point out that this estimate can be obtained immediately by the standard interpolation techniques, but the
constant will not be explicit.
\begin{lem}\label{lem2}
If $\p \Omega \in C^{2+\alpha}$, $f\in C^\alpha(\overline{\Omega})$, $\phi\in C^2(\overline{\Omega})$, then there exists a constant $C$ which is independent of ${\vep}$ such that 
\begin{equation}\label{lem2eq12}
    \|u-u_{\vep}\|_{H^1(\Omega)} \le C\sqrt{{\vep}},
\end{equation}
where $C = \sqrt{2C^*(C^*+1)|\Omega|}$,  $C^*=\|(-\Delta \phi -f)^+\|_{L^\infty}$.
\end{lem}

\begin{proof}
Set $v = \max\{u_{\vep},\phi\}$, Apparently, $v\in K$, hence we recall \re{lemn1} to obtain
\begin{equation}\label{lem2eq2}
    \int_\Omega \nabla u \cdot \nabla (\max\{u_{\vep},\phi\} - u) \dif x \ge \int_\Omega f\cdot \big(\max\{u_{\vep},\phi\}- u\big) \dif x.
\end{equation}
The LHS of \re{lem2eq2} can be rewritten as 
\begin{equation*}
\begin{split}
    \int_\Omega \nabla u \cdot \nabla (\max\{u_{\vep}, \phi\} - u) &= \int_{\Omega \cap \{u_{\vep} \ge \phi\}} \nabla u \cdot \nabla (u_{\vep} - u) \dif x + \int_{\Omega \cap \{u_{\vep} < \phi\}} \nabla u \cdot \nabla (\phi - u)\dif x \\
    &= \int_\Omega \nabla u \cdot \nabla (u_{\vep}-u) \dif x + \int_{\Omega \cap \{u_{\vep} < \phi\}} \nabla u \cdot \nabla (\phi-u_{\vep}) \dif x.
    \end{split}
\end{equation*}
A similar manner is also applied to the RHS of \re{lem2eq2}. As a result, we obtain 
\begin{equation}
    \label{lem2eq3}
    \int_\Omega \nabla u \cdot \nabla (u_{\vep} - u) \dif x \ge \int_\Omega f(u_{\vep} - u) \dif x + \int_{\Omega \cap \{u_{\vep} < \phi\}} f(\phi - u_{\vep})\dif x - \int_{\Omega \cap \{u_{\vep} < \phi\}} \nabla u \cdot \nabla (\phi - u_{\vep}) \dif x.
\end{equation}

On the other hand, we now turn out attention to \re{model2}. We multiply both sides of \re{model2} with the test function $u_{\vep} - u$ and integrate over $\Omega$. From integration by parts, we obtain
From integration by parts, we obtain
\begin{equation}
    \label{lem2eq4}
    \int_\Omega \nabla u_{\vep} \cdot \nabla (u_{\vep} - u) \dif x  = \int_\Omega f (u_{\vep} - u) \dif x + \int_\Omega \beta_{\vep} (\phi - u_{\vep}) (u_{\vep} - u) \dif x.
\end{equation}
Subtracting \re{lem2eq2} from \re{lem2eq4} and applying integration by parts again, 
we derive
\begin{equation*}
\begin{split}
    \int_\Omega |\nabla(u_{\vep} - u)|^2 \dif x &\le \int_\Omega \beta_{\vep}(\phi - u_{\vep}) (u_{\vep} - u) \dif u
    + \int_{\Omega \cap \{u_{\vep} < \phi\}} \Big(-f(\phi - u_{\vep}) + \nabla u \cdot \nabla (\phi - u_{\vep})\Big) \dif x\\
    &= \int_\Omega \beta_{\vep} (\phi - u_{\vep}) (\phi - u_{\vep}) \dif x - \int_\Omega f(\phi - u_{\vep})^+ \dif x + \int_\Omega \nabla u\cdot \nabla((\phi-u_{\vep})^+) \dif x. \\
    &= \int_\Omega \beta_{\vep} (\phi - u_{\vep}) (u_{\vep} - u) \dif x + \int_\Omega (-\Delta u -f) (\phi-u_{\vep})^+ \dif x \\
    &\le \int_\Omega \beta_{\vep} (\phi - u_{\vep}) (u_{\vep} - u) \dif x + \int_\Omega (-\Delta u -f) (u-u_{\vep})^+ \dif x,
    \end{split}
\end{equation*}
where the last inequality is based on the fact that $\phi - u \le 0$ so that
\begin{equation*}
    (\phi - u_{\vep})^+ \le \big( \phi - u + u - u_{\vep}\big)^+ 
    \le (u-u_{\vep})^+.
\end{equation*}
Therefore, we have
\begin{equation*}
\begin{split}
    \int_\Omega |\nabla(u_{\vep} - u)|^2 \dif x
    &\le \int_\Omega \|\beta_{\vep}\|_{L^\infty} \|u_{\vep} - u\|_{L^\infty} \dif x + \int_\Omega \|(-\Delta u -f)^+\|_{L^\infty} \|u_{\vep} - u\|_{L^\infty} \dif x\\
    &= \big(\|\beta_{\vep}\|_{L^\infty} + \|(-\Delta u -f)^+\|_{L^\infty}\big)\|u_{\vep}-u\|_{L^\infty}|\Omega|.
    \end{split}
\end{equation*}
Combining it with the estimates in Step 2 (which implies that 
$\|(-\Delta u -f)^+\|_{L^\infty} \le C^*$)
and \re{ueps} in Lemma \ref{lem1}, we further have
\begin{equation*}
    \int_\Omega |\nabla(u_{\vep} - u)|^2 \dif x \le (2C^*)(C^*+1) |\Omega|\vep,
\end{equation*}
which implies \re{lem2eq12} by taking $C = \sqrt{2C^*(C^*+1)|\Omega|}$.
\end{proof}

After we derive the estimates of $\|u - u_{\vep}\|_{H^1(\Omega)}$, let us consider $\|u- U\zeta\|_{H^1(\Omega)}$. By \textit{Theorem 3} of \cite{hornik1991approximation} and the smoothness of the cutoff function $\zeta$, we have for any small $\delta_4 > 0$, there is a sum of the form \re{nn1} such that 
\begin{equation}\label{m2n1}
    \|U\zeta - u_{\vep}\|_{H^1(\Omega)} \le \|U\zeta - u_{\vep}\|_{C^1(\Omega)} \le C\delta_4,
\end{equation}
if the activation function $\sigma \in C^1(\mathbb{R})$ is nonconstant and bounded. The constant $C$ in \re{m2n1} is independent of $\delta_4$. Following the similar procedures as in \re{nne}, and taking $\vep = O(\delta_4^2)$, we shall get the following theorem.

\begin{thm}\label{thm3.2}
If $\p \Omega\in C^{2+\alpha}$, $f\in C^\alpha(\overline{\Omega})$, $\phi\in C^2(\overline{\Omega})$, and we assume the activation function $\sigma$ is continuously differentiable, nonconstant, and bounded, then for any $\delta_4 > 0$ and $\vep = O(\delta_4^2)$, there exists a function $U$ implemented by \re{nn1} such that 
\begin{equation*}
    \|u - U\zeta\|_{H^1(\Omega)} \le C\delta_4,\mbox{ or equivalently \,}  \|u - U\zeta\|_{H^1(\Omega)} \le C\sqrt{\vep},
\end{equation*}
where the constant $C$ is independent of $\delta_4$. 
\end{thm} 

Again, together with the \textit{Corollary 1} from \cite{siegel2020approximation}, we have the following approximation rate when the activation function is $\text{ReLu}^2$:

\begin{thm}\label{thm3.6}
If the assumptions in Theorem \ref{thm3.2} hold, and 
let $\sigma(x) = ReLu^2(x) = \max\{0,x\}^2$, then there exists a function $U$ implemented by \re{nn1} such that
\begin{equation*}
    \|u - U\zeta\|_{H^1(\Omega)} \le C N^{-\frac12}\|u_{\vep}\|_{B^2},
\end{equation*}
where the constant $C$ is independent of $N$, and $\|u_{\vep}\|_{B^2} = \int_{\mathbb{R}^d} (1+|\omega|)^2 |\hat{u}_{\vep} (\omega)|\dif \omega$ denotes the Barron norm of the function $u_{\vep}$.
\end{thm}

\subsection{Homotopy algorithm}
In Method 2, the systems would become singular as $\vep$ is close to 0. To tackle the singularity and make the approximation more precise, we apply a homotopy algorithm in the optimization problem. 
The basic idea is to train a simple model at the beginning, then to increase the complexity of the structure, and finally to train the original model. 
More specifically, we introduce a homotopy parameter $t\in[0,1]$ in the loss function \re{loss} and consider 
\begin{equation}
    \label{loss1}
    F(\theta,t) = \int_\Omega \Big(\frac12|\nabla (U\zeta)|^2 - U\zeta f + t B_{\vep}(\phi-U\zeta)\Big) \dif x.
\end{equation}
From the previous analysis, we know that minimizing $F(\theta,t)$ is equivalent to solving a Poisson's equation
\begin{eqnarray*}
    &-\Delta u - f = t B_{\vep}'(\phi-u) \hspace{2em} &\text{in }\Omega,\\
    &u=0 \hspace{2em} &\text{on }\p\Omega.
\end{eqnarray*}
When $t=0$, the right-hand side of the first equation is 0; in this case, the system is linear, hence the solution can be easily approximated by the optimization problem with random initial value for the neural networks. As we gradually increase $t$, the impact of the penalized term is also increasing. For each $t=i\delta t$ ($\delta t$ is the homotopy stepsize, and $i\ge 1$), we solve the optimization problem by using the result of $\text{argmin}_\theta F(\theta,(i-1)\delta t)$ as the initial guess of the parameters. When $t=1$, we recover the problem \re{min}. The detailed algorithm is shown in Algorithm \ref{alg}.

\begin{algorithm}[H]
\caption{The homotopy algorithm}
\label{alg}
\begin{algorithmic}
\STATE Solve the optimization problem $\min_\theta F(\theta,0)$ and denote the solution as $\theta^0$;
\STATE Set $N = \frac{1}{\delta t}$, where $\delta t$ is the homotopy stepsize;
\FOR{$i=1,\cdots,N$}
\STATE Solve $\theta^i = \text{argmin}_\theta F(\theta, i \delta t)$ by using $\theta^{i-1}$ as the initial guess;
%\STATE Generate $m$ random points ${\bm \hat{\bm \theta}}_k = (\hat{\theta}_{k,i})_{i=1}^m$ from $[0,2\pi]$;\ %according to the probability density $\nu$;\
%\STATE Calculate the loss function $F(\mathcal{X}_k, t)$;\
%\STATE Compute a stochastic vector $G(\mathcal{X}_k, {\bm \hat{\bm \theta}}_k) = \nabla_\mathcal{X} F(\mathcal{X}_k,{\bm \hat{\bm \theta}}_k)$;\
%\STATE Set the new iterate as $\mathcal{X}_{k+1} = \mathcal{X}_k - \alpha_n G(\mathcal{X}_k,{\bm \hat{\bm \theta}}_k)$;
\ENDFOR
\end{algorithmic}
\end{algorithm}

\section{Simulation results}
\subsection{Example 1: a 1-D problem}

For our first example, we start our investigation of the numerical result for the following simple one dimensional obstacle problem
\begin{eqnarray*}
    &&-u''\ge 0 \hspace{2em} -2<x<2,\\
    &&u\ge 1-x^2 \hspace{2em} -2<x<2,\\
    &&-u''\{u-(1-x^2)\}=0 \hspace{2em} -2<x<2,\\
    &&u(-2)=u(2)=0.
\end{eqnarray*}
The exact solution of the above system can be solved explicitly
\begin{equation*}
    u_{\text{exact}} = \left\{ 
    \begin{array}{ll}
    (4-2\sqrt{3})(x+2) \hspace{2em} &-2\le x\le -2+\sqrt{3},\\
    (1-x^2)\hspace{2em} &-2+\sqrt{3}\le x\le 2-\sqrt{3},\\
    (4-2\sqrt{3})(2-x) \hspace{2em} &2-\sqrt{3}\le x\le 2.
    \end{array}
    \right.
\end{equation*}

For both methods, we set the activation function $\sigma(x) = \max\{0,x\}^2$ and the maximum number of iterations as 4000. The initial values are randomly picked. With different choices of the number of neurons $N$ and $\vep$, we ran both Method 1 and Method 2 ten times, and calculated the average error with the exact solution. The results are shown in Tables \ref{T1} and \ref{T2}. 

In the first table, we found that the errors decrease as we increase the number of neurons $N$ from 10 to 40, and we can compute the convergence rate of $N$ by using the formula
\begin{equation}\label{rate1}
    \log_{\frac12}\Big(\frac{\text{error with $N=10$} - \text{error with $N=20$}}{\text{error with $N=20$}-\text{error with $N=40$}}\Big).
\end{equation}
It is computed that the error is about $O(N^{-0.61})$, which is very close to the rate we obtained in Theorem \ref{ratem1}. We want to emphasize that the rates in Theorems \ref{ratem1} and \ref{thm3.6} might not be the optimal rate of $N$. On the other side, since the optimization problem is highly non-linear and non-convex, particularly as $N$ gets bigger, then the solution may stuck at a local minimum (not a global minimum). Due to this reason, we see that the errors become even larger when $N=80$. In order to obtain the optimal rate numerically, efforts are required to tune hyperparameters properly; for instance, we can use a greedy algorithm instead of gradient descent, which dynamically builds the neural network starting from a simplified version and solves a sub-optimization problem; in this way, a global minimum is guaranteed at each iteration \cite{hao2021efficient}. In our future work, we will consider the implement of a greedy algorithm in this project.

As for Method 2, we first kept $N$ fixed and varied $\vep$. We see that as $\vep$ decreases, the errors become smaller. The convergence rate of $\vep$ can be computed by
\begin{equation}\label{rate2}
    \log_{10}\Big(\frac{\text{error with $\vep=0.1$} - \text{error with $\vep=0.01$}}{\text{error with $\vep=0.01$}-\text{error with $\vep=0.001$}}\Big).
\end{equation}
The result is about $O(\vep)$, and it justifies the rate in Theorem \ref{thmm2}. Moreover, when $\vep$ is fixed, the errors decrease as $N$ gets bigger. The convergence rate of $N$ can be similarly computed --- it is around $O(N^{-0.82})$. %Similarly, we could not obtain $O(N^{-1/2})$ unless we use a greedy training algorithm in the optimization. 
In addition, we also applied the Method 2 with homopoty (we let $\delta t =0.1$ in Algorithm \ref{alg}),  and we obtained a more accurate estimate for the solution. We see that the $L^\infty$ error in this case is roughly within the $2\vep$ range, which agrees well with the estimate \re{nne}.

\begin{table}[H]
\begin{tabular}{ | m{4em} | m{5cm}| m{5cm} |} 
\hline
   &  Method 1 ($L^\infty$ norm) & Method 1 ($L^2$ norm) \\
   \hline
   $N=10$ & $1.021\cdot 10^{-2}$ & $7.275\cdot 10^{-5}$\\
   \hline
   $N=20$ & $7.203\cdot 10^{-3}$ & $5.374\cdot 10^{-5}$\\
   \hline
   $N=40$ & $5.241\cdot 10^{-3}$ & $3.602\cdot 10^{-5}$\\
   \hline
   $N=80$ & $1.724\cdot 10^{-2}$ & $1.247\cdot 10^{-4}$\\
  \hline
\end{tabular}
\caption{Numerical errors of Method 1 for Example 4.1.}\label{T1}
\end{table}

\begin{table}[H]
\begin{tabular}{ | m{8em} | m{3.5cm}| m{3.5cm} | m{3.5cm} |} 
\hline
   &  Method 2 ($L^\infty$ norm) & Method 2 ($L^2$ norm) & Method 2 with homotopy ($L^\infty$ norm) \\
   \hline
   $\vep=0.1$, $N=20$ & $2.243\cdot 10^{-1}$ & $1.878\cdot 10^{-3}$ & $2.198\cdot 10^{-1}$ \\
   \hline
   $\vep=0.01$, $N=20$ & $3.380\cdot 10^{-2}$ & $2.327\cdot 10^{-4}$ & $2.943\cdot 10^{-2}$\\
   \hline
   $\vep=0.001$, $N=20$ & $1.594\cdot 10^{-2}$ & $1.191\cdot 10^{-4}$ & $2.964\cdot 10^{-3}$\\
   \hline
   $\vep = 0.001$, $N=10$ & $1.978\cdot 10^{-2}$ & $1.433\cdot 10^{-4}$ & $7.496\cdot 10^{-3}$\\
  \hline
  $\vep = 0.001$, $N=40$ & $1.376\cdot 10^{-2}$ & $1.034\cdot 10^{-4}$ & $3.133\cdot 10^{-3}$\\
  \hline
\end{tabular}
\caption{Numerical errors of Method 2 for Example 4.1.}\label{T2}
\end{table}

In Figures  \ref{fig-M1-1D} -- \ref{fig-M2-1Dwh}, we present one simulation result obtained by using Method 1, Method 2, and Method 2 with homotopy, respectively. 
In these three figures, the left plot shows the numerical solution we obtained, and the right plot exhibit the error compared with the exact solution.
We can see that the numerical solutions approximate the exact solution very well in the all three cases.

\begin{figure}[H]
\begin{subfigure}{.48\textwidth}
  \centering
  % include first image
  \includegraphics[height=2in]{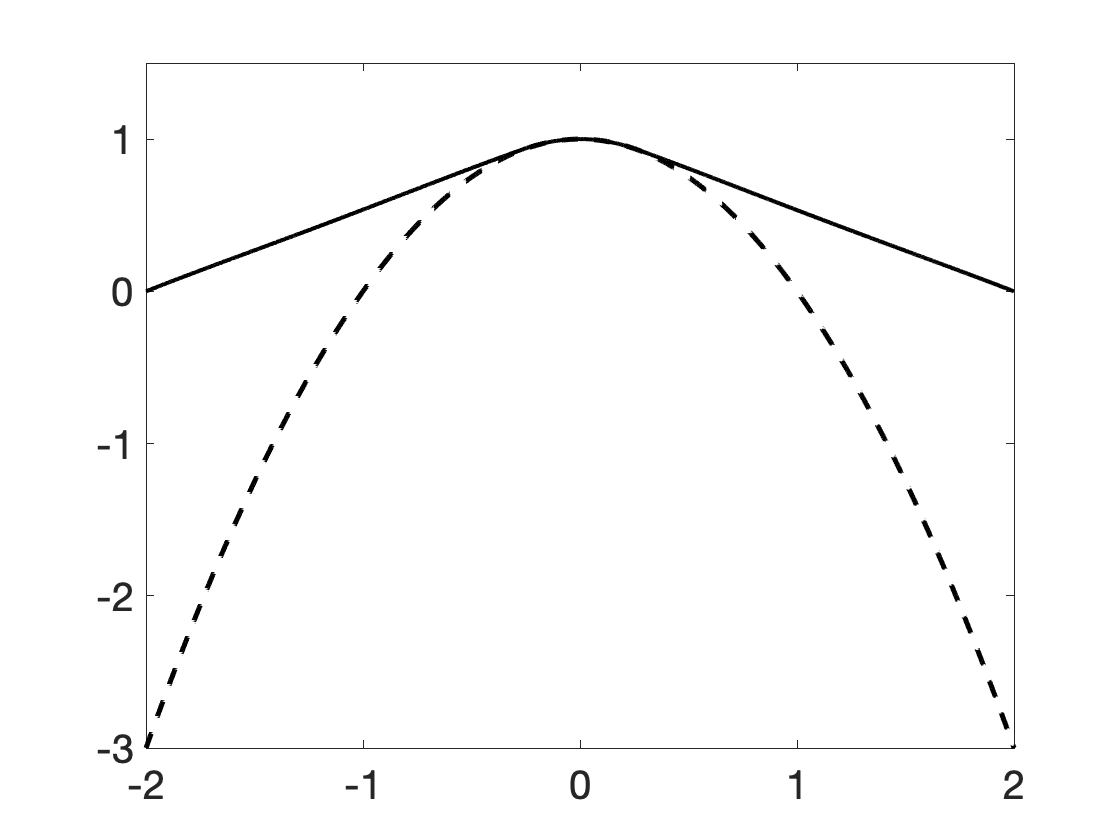}
\end{subfigure}
\begin{subfigure}{.48\textwidth}
  \centering
  % include second image
  \includegraphics[height=2in]{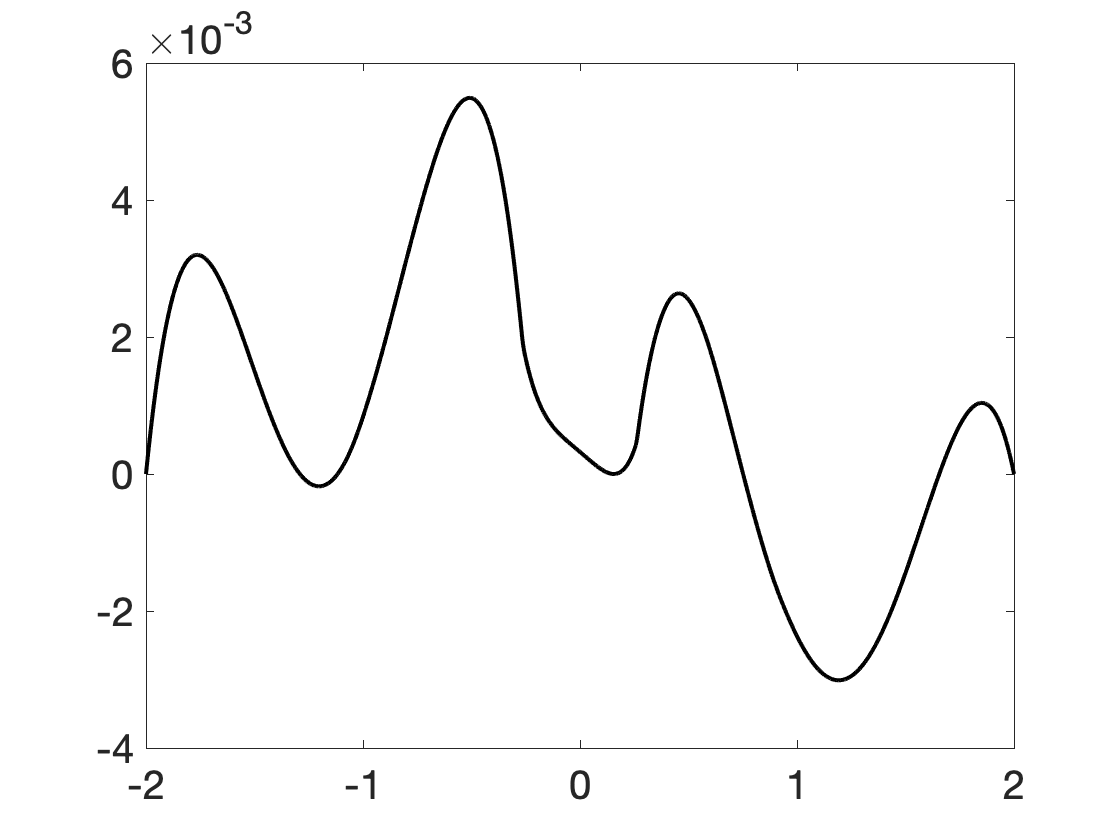}  
\end{subfigure}
\caption{Results obtained by using Method 1, $N=20$: (Left) the dashed curve is the obstacle and the black one is our numerical solution; (Right) the difference between our numerical solution and the exact solution.}
\label{fig-M1-1D}
\end{figure}

\begin{figure}[H]
\begin{subfigure}{.48\textwidth}
  \centering
  % include first image
  \includegraphics[height=2in]{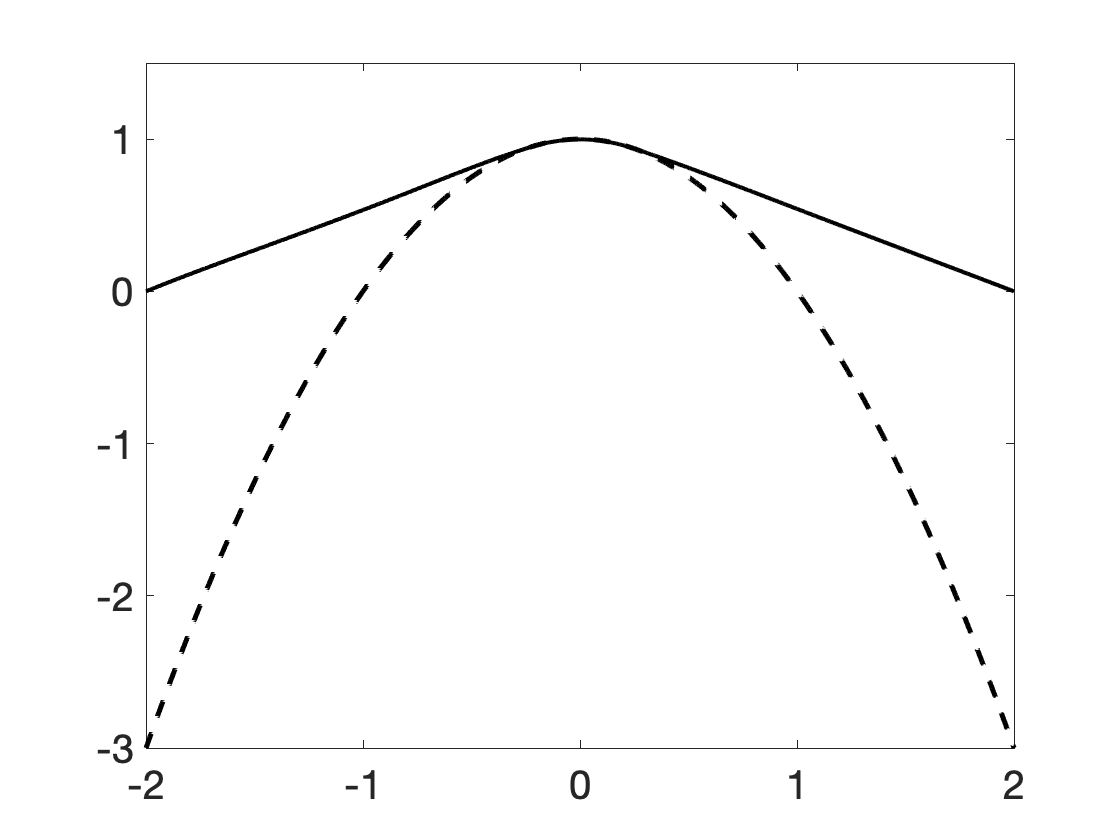}
\end{subfigure}
\begin{subfigure}{.48\textwidth}
  \centering
  % include second image
  \includegraphics[height=2in]{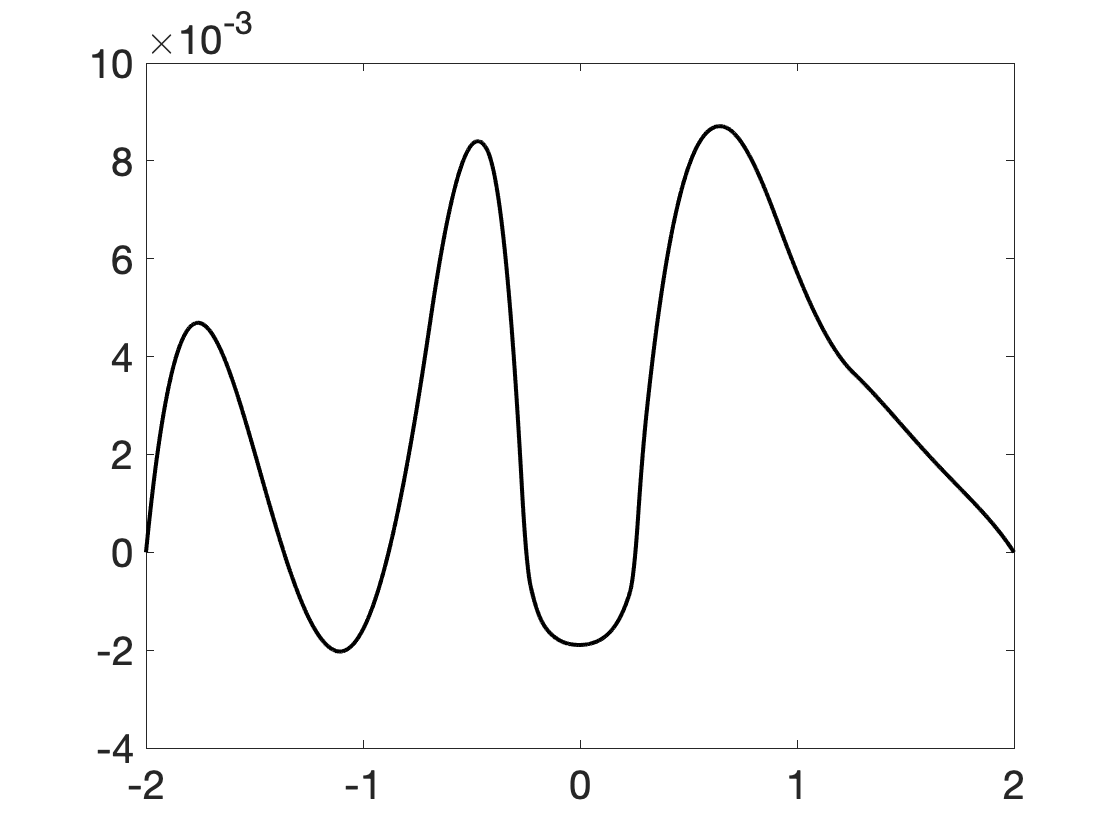}  
\end{subfigure}
\caption{Results obtained by using Method 2, $\vep=10^{-3}$, $N=20$: (Left) the dashed curve is the obstacle and the black one is our numerical solution; (Right) the difference between our numerical solution and the exact solution.}
\label{fig-M2-1Dwh}
\end{figure}

\begin{figure}[H]
\begin{subfigure}{.48\textwidth}
  \centering
  % include first image
  \includegraphics[height=2in]{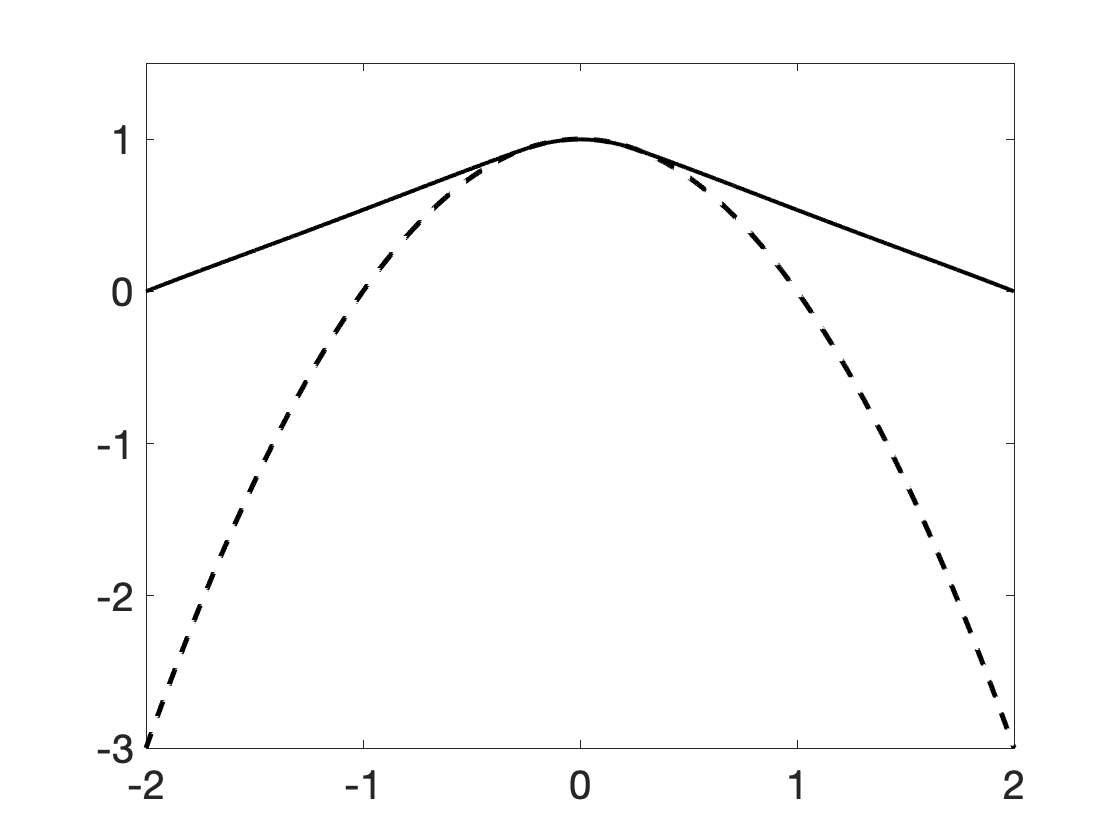}
\end{subfigure}
\begin{subfigure}{.48\textwidth}
  \centering
  % include second image
  \includegraphics[height=2in]{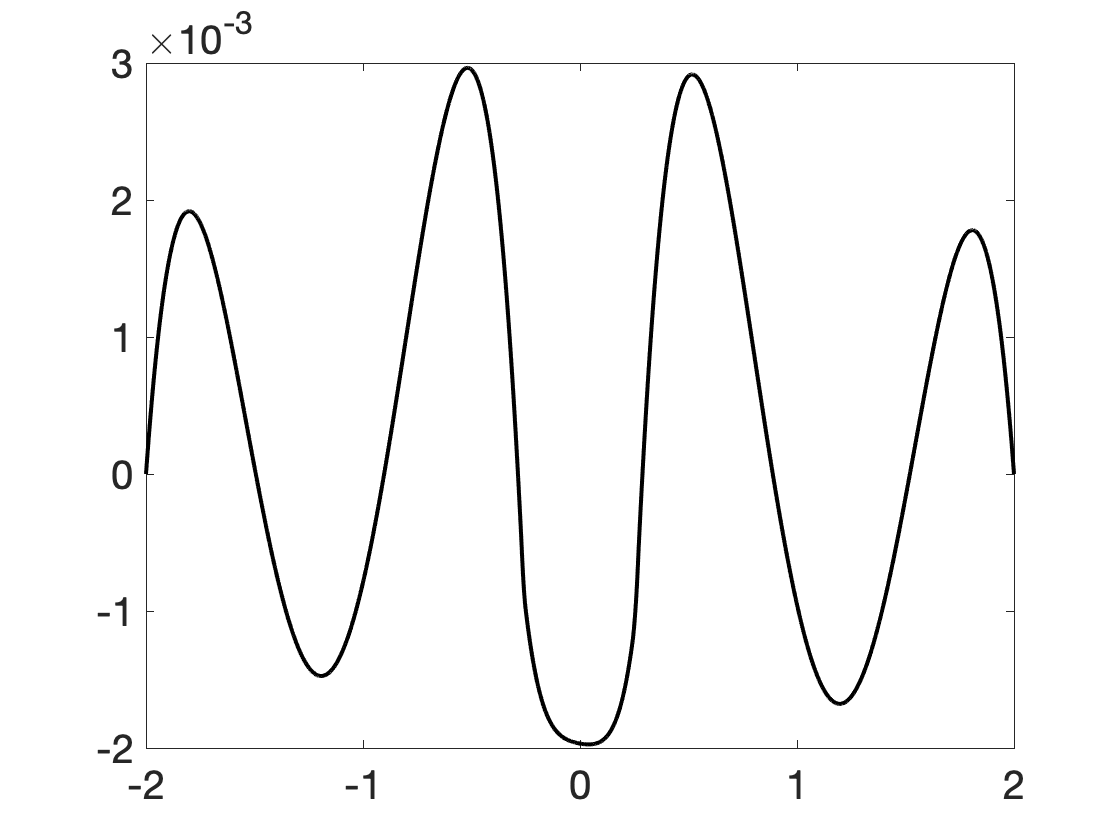}  
\end{subfigure}
\caption{Results obtained by using Method 2 with homotopy, $\vep=10^{-3}$, $N=20$: (Left) the dashed curve is the obstacle and the black one is our numerical solution; (Right) the difference between our numerical solution and the exact solution.}
\label{fig-M2-1D}
\end{figure}

\subsection{Example 2: a 2-D problem}
Next, we consider a two-dimensional (2D) problem on the domain $\Omega = [-2,2]\times [-2,2]$ with $f = -\frac{(r^*)^2}{2\sqrt{1-(r^*)^2}}$ and the obstacle
\begin{equation*}
    \phi(x,y)= \left\{
    \begin{array}{ll}
    \sqrt{1-x^2-y^2}-\frac{(r^*)^2}{\sqrt{1-(r^*)^2}}\Big[-\frac18(x^2+y^2)+\frac12\Big] \hspace{2em} &\text{for } x^2+y^2 \le 1,\\
    -1-\frac{(r^*)^2}{\sqrt{1-(r^*)^2}}\Big[-\frac18(x^2+y^2)+\frac12\Big] \hspace{2em} &\text{otherwise},
    \end{array}
    \right.
\end{equation*}
where $r^*$ is a constant which satisfies $(r^*)^2(1-\log(r^*/2))=1$. The analytical solution to the above system can be solved directly as
\begin{equation*}
    u_{\text{exact}} = \left\{
    \begin{array}{ll}
    \sqrt{1-x^2-y^2}-\frac{(r^*)^2}{\sqrt{1-(r^*)^2}}\Big[-\frac18(x^2+y^2)+\frac12\Big] \hspace{2em} &\text{for } x^2+y^2 \le (r^*)^2,\\
    -\frac{(r^*)^2}{\sqrt{1-(r^*)^2}}\log\frac{\sqrt{x^2+y^2}}{2}-\frac{(r^*)^2}{\sqrt{1-(r^*)^2}}\Big[-\frac18(x^2+y^2)+\frac12\Big] \hspace{2em} &\text{for } (r^*)^2\le x^2+y^2\le 2^2,\\
    0\hspace{2em} &\text{for } x^2+y^2\ge 2^2.
    \end{array}
    \right.
\end{equation*}

Similar as in the Example 1, we used the ReLu2 function, $\max\{0,x\}^2$, as our activation function, took random values for the initial parameters, and set the maximum number of iterations as 4000. The average errors based on 10 simulation results are shown in Tables \ref{T3} and \ref{T4}, and the plots are exhibited in Figures \ref{fig-M1-2D} --- \ref{fig-M2-2Dwh}.

Using the data from Table \ref{T3}, we can compute the convergence rate of Method 1 based on the formula \re{rate1}; it is approximately $O(N^{-0.50})$, which is consistent with the theoretical derived order in Theorem \ref{ratem1}. For the Method 2, we see from the first three rows of Table \ref{T4} that the errors first decrease as we decrease $\vep$, but increase slightly when we further decrease $\vep$ to 0.001; it indicates that the dominant error in this case is the error induced by neural networks. In fact, from the second row and the last two rows of Table \ref{T4}, we found that if $\vep$ is fixed, the error is smaller with larger $N$, and the method achieves a convergence order of $O(N^{-1.89})$ with our results. We want to emphasize that, in the convergence analysis of Method 2, we have chosen $\vep = O(\delta_4^2)=O((N^{-0.5})^2)=O(1/N)$ to obtain the convergence rate Theorem \ref{thm3.6}. Therefore, in the simulations, we need to carefully tuned the values of $\vep$ and $N$ such that the rate in Theorem \ref{thm3.6} can be achieved.

\begin{table}[H]
\begin{tabular}{ | m{4em} | m{5cm}| m{5cm} |} 
\hline
   &  Method 1 ($L^\infty$ norm) & Method 1 ($L^2$ norm) \\
   \hline
   $N=10$ & $8.864\cdot 10^{-2}$ & $2.007\cdot 10^{-4}$\\
   \hline
   $N=20$ & $7.008\cdot 10^{-2}$ & $1.530\cdot 10^{-4}$\\
   \hline
   $N=40$ & $5.700\cdot 10^{-2}$ & $1.250\cdot 10^{-4}$\\
  \hline
\end{tabular}
\caption{Numerical errors of Method 1 for Example 4.2.}\label{T3}
\end{table}

\begin{table}[H]
\begin{tabular}{ | m{8em} | m{3.5cm}| m{3.5cm} | m{3.5cm} |} 
\hline
   &  Method 2 ($L^\infty$ norm) & Method 2 ($L^2$ norm) & Method 2 with homotopy ($L^\infty$ norm) \\
   \hline
   $\vep=0.1$, $N=20$ & $2.868\cdot 10^{-1}$ & $7.164\cdot 10^{-4}$ & $2.860\cdot 10^{-1}$ \\
   \hline
   $\vep=0.01$, $N=20$ & $5.190\cdot 10^{-2}$ & $1.249\cdot 10^{-4}$ & $4.385\cdot 10^{-2}$\\
   \hline
   $\vep=0.001$, $N=20$ & $6.085\cdot 10^{-2}$ & $1.423\cdot 10^{-4}$ & $2.448\cdot 10^{-2}$\\
   \hline
   $\vep = 0.01$, $N=10$ & $5.972\cdot 10^{-2}$ & $1.392\cdot 10^{-4}$ & $4.509\cdot 10^{-2}$\\
  \hline
  $\vep = 0.01$, $N=40$ & $4.592\cdot 10^{-2}$ & $1.096\cdot 10^{-4}$ & $4.331\cdot 10^{-2}$\\
  \hline
\end{tabular}
\caption{Numerical errors of Method 2 for Example 4.2.}\label{T4}
\end{table}

\begin{figure}[H]
\begin{subfigure}{.48\textwidth}
  \centering
  % include first image
  \includegraphics[height=2in]{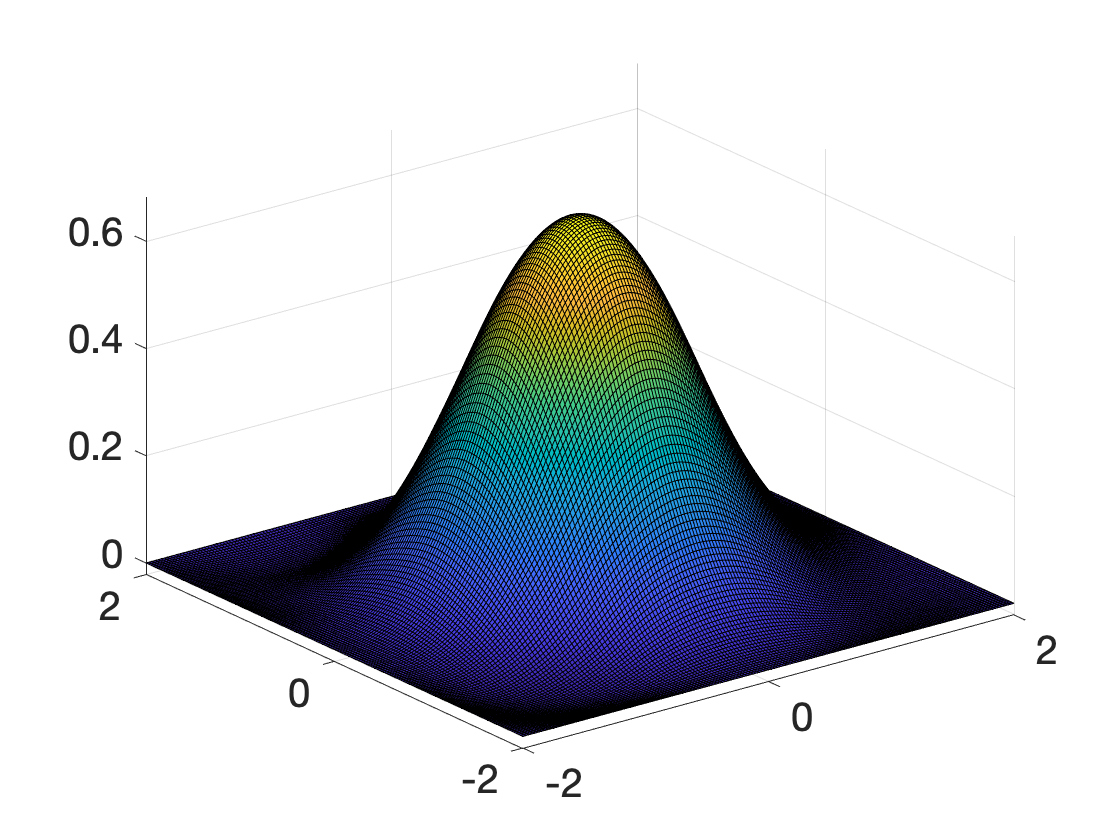}
\end{subfigure}
\begin{subfigure}{.48\textwidth}
  \centering
  % include second image
  \includegraphics[height=2in]{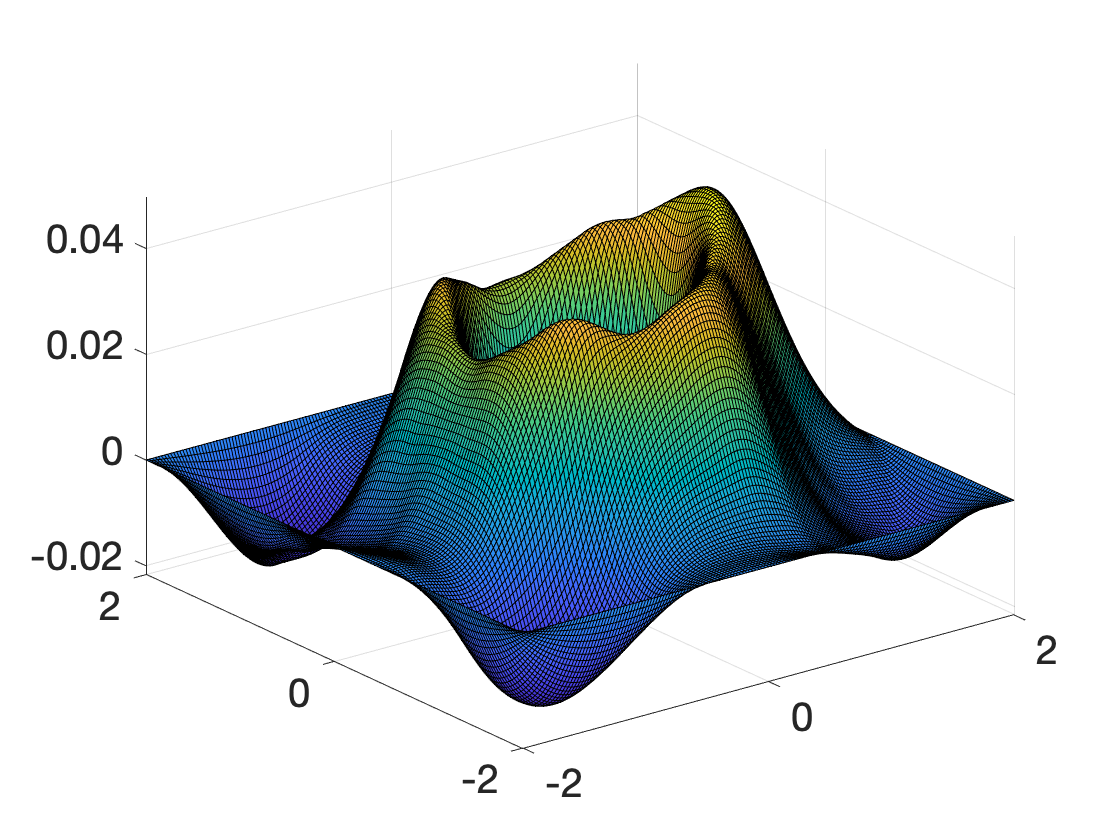}  
\end{subfigure}
\caption{Results obtained by using Method 1, $N=40$: (Left) the numerical solution; \\(Right) the difference with the analytical solution.}
\label{fig-M1-2D}
\end{figure}

\begin{figure}[H]
\begin{subfigure}{.48\textwidth}
  \centering
  % include first image
  \includegraphics[height=2in]{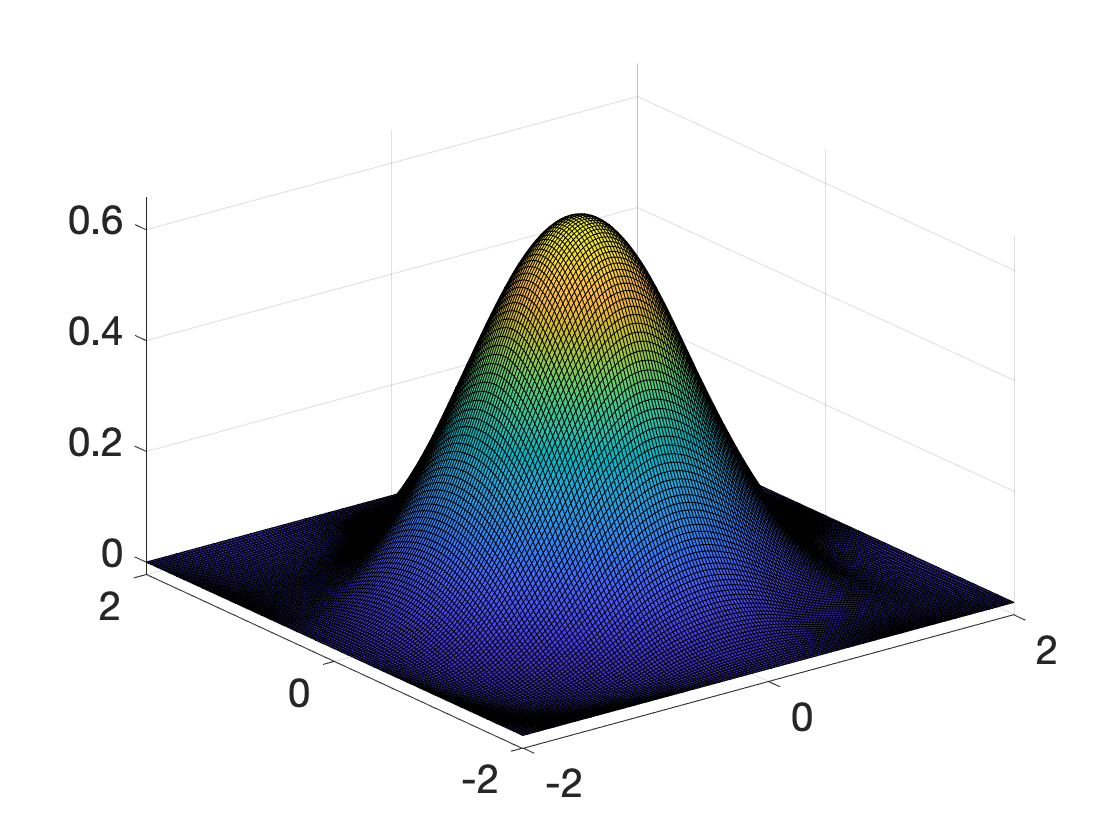}
\end{subfigure}
\begin{subfigure}{.48\textwidth}
  \centering
  % include second image
  \includegraphics[height=2in]{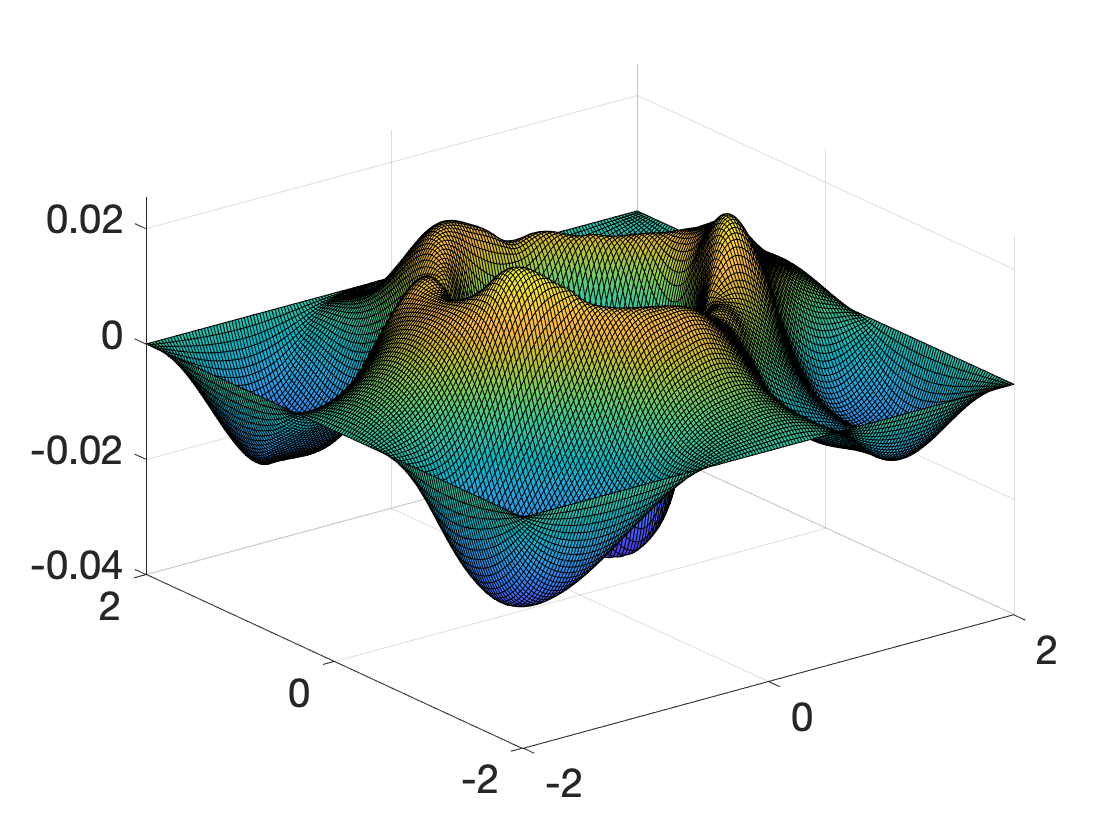}  
\end{subfigure}
\caption{Results obtained by using Method 2, $\vep=0.01$, $N=40$: (Left) the numerical solution; (Right) the difference with the analytical solution.}
\label{fig-M2-2D}
\end{figure}

\begin{figure}[h]
\begin{subfigure}{.48\textwidth}
  \centering
  % include first image
  \includegraphics[height=2in]{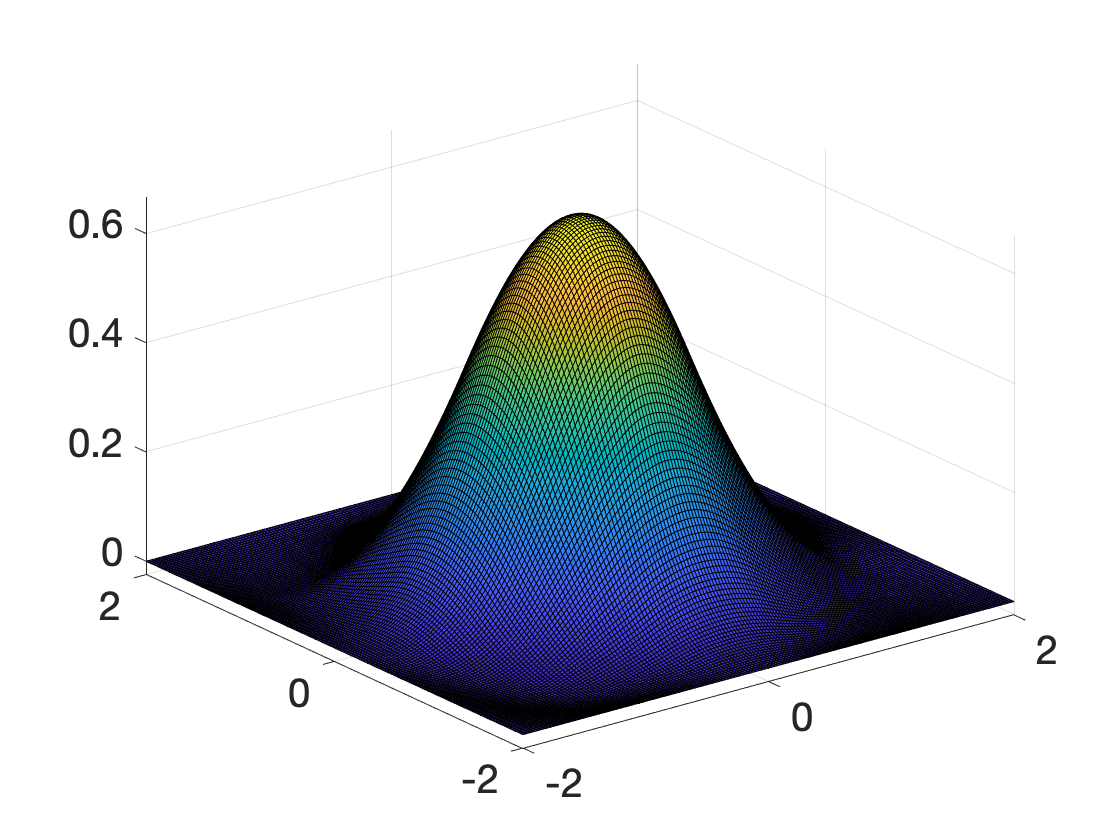}
\end{subfigure}
\begin{subfigure}{.48\textwidth}
  \centering
  % include second image
  \includegraphics[height=2in]{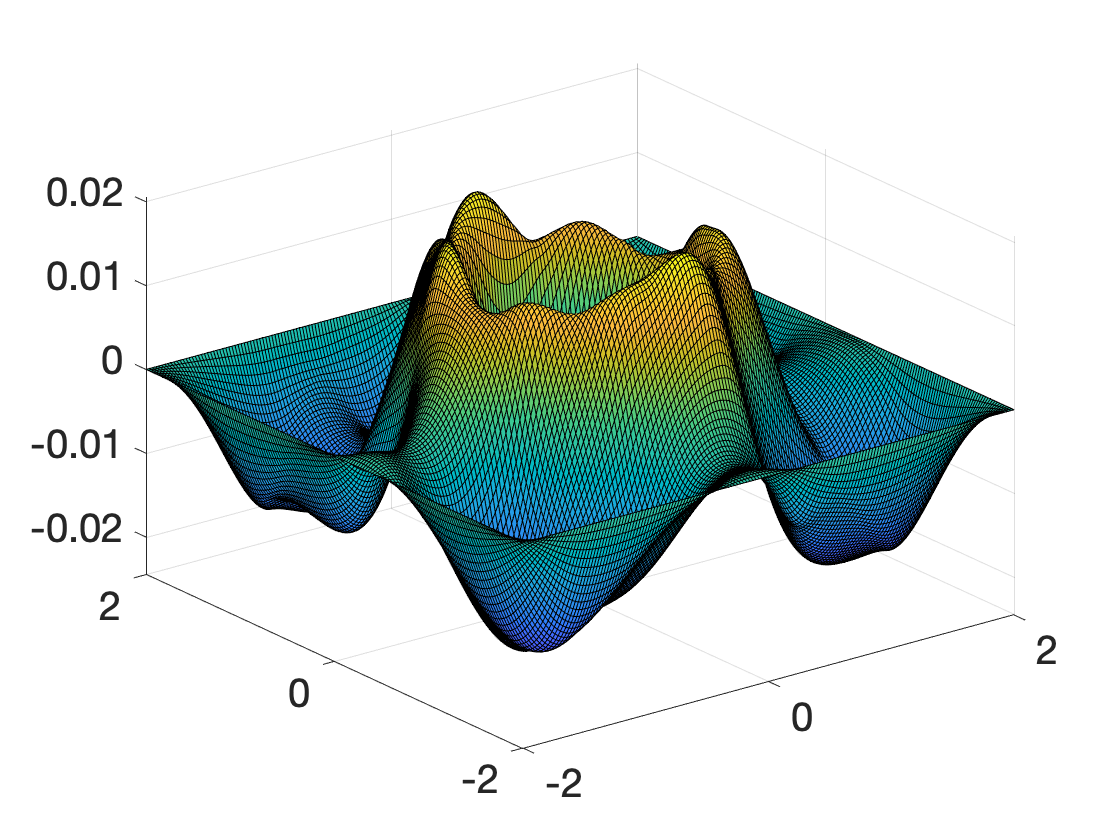}  
\end{subfigure}
\caption{Results obtained by using Method 2 with homotopy, $\vep=0.001$, $N=20$: (Left) the numerical solution; (Right) the difference with the analytical solution.}
\label{fig-M2-2Dwh}
\end{figure}
\section{Conclusion}
Two neural-network-based schemes were proposed for solving free boundary problems coming from contact with obstacles. The cost function of the first method comes naturally from the energy minimization of the problem, and the cost function of the second method is based on the variational form of a penalized system. These two new methods were verified on a 1-D and a 2-D obstacle problem, and the convergence rates were computed based on the simulation results. While the numerical solutions obtained by using these two methods approximate the actual solutions very well, the theoretical convergence rates were not achieved. This is because the training algorithm we used does not guarantee the global minimum. In our future work, we will try other training algorithms, e.g., greedy training algorithm \cite{hao2021efficient}, on our methods to see whether the theoretical rates can be recovered in our problem.

\end{document}